\documentclass{article}
\usepackage[margin=2.5cm]{geometry}

\usepackage{graphicx}
\usepackage{url}
\usepackage{booktabs}
\usepackage{placeins}
\graphicspath{{figs/}}

\usepackage{amsmath,amsfonts,amssymb}
\usepackage{amsthm}
\newtheorem{theorem}{Theorem}[section]
\newtheorem{proposition}[theorem]{Proposition}
\newtheorem{lemma}[theorem]{Lemma}

\theoremstyle{definition}
\newtheorem{definition}[theorem]{Definition}

\theoremstyle{remark}
\newtheorem{remark}[theorem]{Remark}

\providecommand{\keywords}[1]{\textbf{Keywords:} #1}

\newcommand{\Torus}{\mathbb{T}}
\newcommand{\R}{\mathbb{R}}
\newcommand{\Z}{\mathbb{Z}}

\newcommand{\dd}{\,\mathrm{d}}
\newcommand{\Id}{\mathbf{1}}

\newcommand{\Hil}{\mathcal{H}}
\newcommand{\Bop}{\mathcal{B}}

\newcommand{\Azero}{\mathcal{A}_0}
\newcommand{\Mu}{\mathcal{M}_u}
\let\Alg\Mu

\newcommand{\AvN}{\mathcal{A}^{(s)}_{\mathrm{vN}}}
\newcommand{\Aadv}{\mathcal{A}^{(s)}_{0}}

\begin{document}

\title{Crossed-Product von Neumann Algebras for Incompressible Navier--Stokes Flows\\
and Spectral Complexity Indicators}

\title{Crossed-Product von Neumann Algebras for Incompressible 
Navier--Stokes Flows\\
and Spectral Complexity Indicators}

\author{Gautier-Edouard Filardo\\[2mm]
\textit{Efrei, Paris Panthéon-Assas}\\
\texttt{gautier-edouard.filardo@efrei.fr}\\[1mm]
\small{https://orcid.org/0009-0005-8310-2065}
}

\date{}

\maketitle

\begin{abstract}
We introduce a traceable operator-algebraic framework for incompressible transport on $M=\Torus^3$ (and, more generally, compact
Riemannian manifolds endowed with a smooth invariant probability measure).
Given an autonomous divergence-free velocity field $u$, the time-$1$ map $\Phi$ induces the Koopman unitary $U$ on $L^2(M)$ and the
crossed-product finite von Neumann algebra
\[
\Mu \;:=\; L^\infty(M)\rtimes_{\alpha}\Z \;=\; W^*(L^\infty(M),U),
\]
equipped with its canonical faithful normal trace $\tau_u$.
Within $\Mu$ we define tracial complexity functionals from commutators $[U,M_f]$ (with $M_f$ the multiplication operators) and associated
positive elements, and we connect these quantities to Fuglede--Kadison determinants and entropy-like tracial functionals.
In parallel, we introduce bounded regularized advection operators $\widetilde T_u^{(s)}:=K_s\,T_u\,K_s$ as differential-level probes of transport
noncommutativity, and we recall the Lie-bracket commutator identity at the formal generator level.
This provides a natural algebraic setting in which tracial invariants are well posed and, in principle, computable on discretizations
(e.g.\ cavity flow and vortex benchmarks).
\end{abstract}


\section{Introduction}\label{sec:intro}

\subsection{From fluid particles to operator algebras: physical motivation}\label{subsec:physical_motivation}

\paragraph{The challenge of turbulent mixing}
In a turbulent fluid, an initially smooth blob of dye is stretched and folded by the flow, developing structure at arbitrarily fine scales.
Classical PDE methods track this complexity through Sobolev norms $\|u(t)\|_{H^s}$, but these can blow up even when the flow remains
physically regular---a phenomenon tied to the \emph{supercriticality barrier} in 3D \cite{Tao2007NS,Tao2016JAMS}.

\paragraph{The noncommutative picture}
Our key observation is that fluid transport is fundamentally \emph{noncommutative}: advecting along velocity field $u$ and then $v$
is not the same as advecting along $v$ then $u$. The commutator $[T_u,T_v]$ measures this failure, and by Theorem~\ref{thm:comm}
it equals $T_{[u,v]}$ where $[u,v]$ is the Lie bracket of vector fields---precisely the \emph{vortex stretching term} in fluid dynamics.

\paragraph{Why operator algebras?}
Classical analysis tracks pointwise values $u(x,t)$. But in turbulence, ``points'' mix and lose identity.
Operator algebras replace the manifold $M$ by a noncommutative algebra generated by advection operators, where:
\begin{itemize}
\item Commutators $[T_u,T_v]$ encode vortex stretching and nonlinear cascade,
\item Traces $\tau_u$ provide scale-invariant averages immune to pointwise blow-up,
\item Fuglede--Kadison determinants detect ``hidden'' structural constraints (e.g.\ anisotropy, conservation laws).
\end{itemize}
This parallels how quantum mechanics replaces phase space $(x,p)$ by noncommuting operators $(\hat X,\hat P)$ when particles
lose classical trajectories \cite{Connes1994}.

\paragraph{Physical intuition for $\mathcal{S}(F)$}
The functional $\mathcal{S}(F)=\int \log(1+\sum|f-f\circ\Phi|^2)\,d\mu$ measures how much observables $f\in F$ are ``scrambled''
by one time-step of the flow $\Phi$ (Theorem~\ref{thm:exact_reduction}).
If $\mathcal{S}(F_K)$ grows unboundedly as we include finer-scale modes $F_K$ (Fourier cutoffs),
the flow is ``infinitely mixing'' at small scales---a signature of potential blow-up (Section~\ref{sec:bridge_classical}).

\paragraph{Grothendieck's topos viewpoint}
Beyond Connes' spectral geometry, we invoke Grothendieck's functorial philosophy \cite{SGA4,MacLaneMoerdijk1994}:
rather than fixing a point-set manifold, we organize the fluid ``space'' as a \emph{topos}---a category of sheaves encoding
multiscale discretizations. This provides a rigorous framework for comparing coarse-grained (numerical) and continuum (analytic) descriptions,
essential for validating computed $\mathcal{S}(F_K)$ against PDE estimates (Section~\ref{sec:numerics}).

\subsection{Clay Millennium formulations}
Let $u:\Omega\times[0,\infty)\to\R^3$ be the velocity field, $p$ the pressure, and $\nu>0$ the viscosity.
The 3D incompressible Navier--Stokes equations read
\begin{equation}\label{eq:NS}
\partial_t u + (u\cdot\nabla)u = -\nabla p + \nu \Delta u,\qquad \nabla\cdot u = 0.
\end{equation}
In Fefferman's official Clay statement \cite{FeffermanClay}, one is asked to establish global existence and smoothness
either on $\R^3$ or on the three-torus $\Torus^3$, or else to produce a finite-time singularity (blow-up) in one of these settings.
\paragraph{Scope of the operator-algebraic construction (autonomous or periodic setting)}
The crossed-product construction of this paper requires a single measure-preserving map $\Phi$ (the time-$1$ map).
Hence, throughout we treat either \emph{autonomous} incompressible flows $u(x)$, or \emph{time-periodic} incompressible flows $u(x,t)$ of period~$1$,
for which the Poincar\'e map $\Phi=\Phi_{1,0}$ is well-defined and measure-preserving.
For general non-autonomous Navier--Stokes fields $u(x,t)$, one must replace the $\Z$-action by an evolution cocycle/groupoid framework
(Section~\ref{subsec:groupoid_extension}).

\subsection{State of the art and the ``supercriticality barrier''}
Leray's foundational work established global weak (finite-energy) solutions in 3D,
but uniqueness and regularity remain open in general \cite{Leray1934}.
A frequently emphasized difficulty is supercriticality with respect to scaling:
the a priori bounds naturally available for Leray solutions are not scale-invariant in 3D,
and classical compactness/energy methods do not prevent concentration at small scales
(see e.g.\ \cite{Tao2007NS,Tao2016JAMS} for an expository discussion).
At the same time, there exist striking global regularity results for special large-data classes,
notably anisotropic regimes in which additional structure can be exploited \cite{CheminGallagherPaicu2011}.

\subsection{Status quaestionis: partial regularity and conditional criteria}\label{subsec:status}
Beyond the existence of Leray weak solutions \cite{Leray1934}, several landmark results clarify what is known and what fails to close the
Millennium gap.
First, suitable weak solutions enjoy \emph{partial regularity}: the possible singular set is small in a precise parabolic Hausdorff sense
(Caffarelli--Kohn--Nirenberg) \cite{CKN1982}.
Second, many \emph{conditional} criteria show that blow-up would force the divergence of a scale-sensitive quantity;
for instance, in the inviscid limit (Euler) the Beale--Kato--Majda criterion ties breakdown to vorticity growth \cite{BKM1984}.
Third, there exist genuinely global regularity results for special large-data regimes, including anisotropic scenarios
in which scale separation is exploited \cite{CheminGallagherPaicu2011}.
The common theme is that \emph{additional structure} (geometric, spectral, probabilistic, or anisotropic) can restore control
beyond the standard energy method, but no currently known mechanism rules out concentration for general 3D data \cite{FeffermanClay,Tao2007NS,Tao2016JAMS}.

\subsection{What is new here: from anisotropy to operator invariants}\label{subsec:novelty}
Our point of departure is to encode the nonlinear transport term by an operator-algebraic object and to treat ``structure''
(e.g.\ anisotropy, scale separation, vortex organization) as constraints on the generated algebra.
Compared with existing PDE approaches, we propose:
(i) a von Neumann algebra viewpoint where commutators measure noncommutativity/complexity of transport,
(ii) a bridge to cyclic cohomology and index-type pairings in the sense of Connes \cite{Connes1994,Loday1998,ConnesMoscovici1990},
and
(iii) a Grothendieck-style functorial layer (topos/motivic heuristics) to organize multiscale discretizations \cite{SGA4,MacLaneMoerdijk1994}.
This paper develops the operator-algebraic foundations and formulates a program of testable conjectures; proving a full Clay solution
is explicitly left open.

\subsection{Goal and contributions of this manuscript}
This paper does \emph{not} claim a complete resolution of the Clay problem.
Instead, it develops a framework in which \emph{conditional} regularity criteria and \emph{testable} conjectures
can be stated in operator-theoretic terms.

The main contributions at this stage are:
\begin{itemize}
\item the finite crossed-product von Neumann algebra $\Alg$ associated with the Koopman dynamics
(Section~\ref{sec:crossed}), and the regularized advection family $\widetilde T_u^{(s)}$
as differential-level probes (Section~\ref{sec:alg}),
\item a structural commutator identity linking $[T_u,T_v]$ to the Lie bracket of vector fields (Theorem~\ref{thm:comm}),
\item a program to express regularity criteria through operator-algebraic quantities
(commutator growth, traces, and Fuglede--Kadison-type determinants) (outlined here; developed in later sections),
\item a computational agenda (vortex rings / knotted vortices) intended to stress-test these criteria.
\end{itemize}

\subsection{Navigation guide for different audiences}\label{subsec:reading_guide}

This paper synthesizes tools from operator algebras, PDE analysis, and numerical methods. To facilitate reading:

\paragraph{For fluid dynamicists}
Start with \S\ref{sec:bridge_classical} (bridge to classical criteria), which connects the tracial functional $\mathcal{S}(F_K)$
to Beale--Kato--Majda \cite{BKM1984} via Theorem~\ref{thm:conditional_reg}.
The key object is defined \emph{explicitly} in Theorem~\ref{thm:exact_reduction} and requires only flow integration + quadrature.
Skip \S\ref{sec:connes} (cyclic cohomology) on first reading; it is programmatic.

\paragraph{For operator algebraists}
The main construction is the crossed product $\Mu=L^\infty(M)\rtimes_\alpha\Z$ in \S\ref{sec:crossed}.
Theorem~\ref{thm:exact_reduction} shows the commutator functional reduces to a classical integral, making the abstract tracial structure
\emph{computationally tractable}---a rare feature in noncommutative geometry.
The cyclic cohomology layer (\S\ref{sec:connes}) is programmatic but points toward index-theoretic applications.

\paragraph{For numerical analysts}
Section~\ref{sec:numerics} provides a complete protocol for evaluating $S_h(F)$ on discretizations.
The algorithm reduces to: (i) flow map approximation (standard ODE solver), (ii) Monte Carlo integration.
Table~\ref{tab:sanity_checks} validates the implementation on benchmarks with closed-form values;
Table~\ref{tab:blowup_diag} tests sensitivity to near-singular behavior.

\paragraph{Interdependence of sections}
\begin{itemize}
\item Core formalism: \S\ref{sec:alg} (regularized operators) $\to$ \S\ref{sec:crossed} (crossed product) $\to$ Theorem~\ref{thm:exact_reduction}.
\item Regularity connection: \S\ref{sec:bridge_classical} builds on Theorem~\ref{thm:exact_reduction}.
\item Numerics: \S\ref{sec:numerics} is self-contained given Theorem~\ref{thm:exact_reduction}.
\item Cyclic cohomology (\S\ref{sec:connes}) is independent and can be skipped.
\end{itemize}

\section{Operator-algebraic framework: regularized advection operators}\label{sec:alg}

\subsection{Function spaces and the advection operator}
Let $M=\Torus^3$ or a smooth compact 3-manifold with volume form $\dd x$.
For a smooth divergence-free vector field $u$, define the (formal) advection operator
\[
T_u f := u\cdot\nabla f .
\]
As a first-order differential operator, $T_u$ is generally unbounded on $L^2(M)$.
To work inside $\Bop(\Hil)$, we introduce a smoothing scale.

\subsection{Bounded regularization}
Let $\Delta$ denote the (positive) Laplace operator on scalar functions on $M$ and set
\[
K_s := (1-\Delta)^{-s/2},\qquad s>0.
\]
For $u\in C^\infty_\sigma(M)$ define the bounded operator
\[
\widetilde T_u^{(s)} := K_s\,T_u\,K_s.
\]
We denote by $\AvN$ the von Neumann algebra generated (in the weak operator topology) by $\{\widetilde T_u^{(s)}\}$,
and by $\Aadv$ the $*$-algebra they generate.

\subsection{Boundedness of the regularized advection operators}\label{subsec:bounded}

Let $M=\Torus^3$ (or a compact smooth Riemannian $3$-manifold without boundary) and let
$\Delta$ be the positive Laplace--Beltrami operator on scalar functions.
For $s>0$, $K_s=(1-\Delta)^{-s/2}$ extends to a bounded map $L^2(M)\to H^s(M)$.

\begin{proposition}[Order criterion for boundedness]\label{prop:boundedTu}
Let $M$ be a compact smooth Riemannian $3$-manifold and set $K_s=(1-\Delta)^{-s/2}$.
Assume $u\in W^{1,\infty}(M;\R^3)$ and define $\widetilde T_u^{(s)}=K_s\,T_u\,K_s$.
If $s\ge \tfrac12$, then $\widetilde T_u^{(s)}$ extends to a bounded operator on $L^2(M)$ and
\[
\|\widetilde T_u^{(s)}\|_{\mathcal B(L^2)} \le C_s\,\|u\|_{W^{1,\infty}}.
\]
\end{proposition}

\begin{proof}
On a compact manifold, $K_s=(1-\Delta)^{-s/2}$ is bounded from $L^2(M)$ to $H^s(M)$.
For $u\in W^{1,\infty}$ the first-order operator $T_u f=u\cdot\nabla f$ is bounded $H^s(M)\to H^{s-1}(M)$ with norm
$\|T_u\|_{H^s\to H^{s-1}}\le C\,\|u\|_{W^{1,\infty}}$.
Finally, $K_s$ is bounded $H^{s-1}(M)\to H^{2s-1}(M)$, and if $s\ge \tfrac12$ then $2s-1\ge 0$ so the embedding
$H^{2s-1}(M)\hookrightarrow L^2(M)$ yields a bounded map $K_s:H^{s-1}(M)\to L^2(M)$.
Composing gives $\widetilde T_u^{(s)}=K_sT_uK_s\in\mathcal{B}(L^2)$ and the stated estimate.Order-$0$ pseudodifferential operators are bounded on $L^2$ by the Calder\'on--Vaillancourt theorem \cite{CalderonVaillancourt1972}.
\end{proof}

\subsection{Commutator identity}
\begin{theorem}\label{thm:comm}
For smooth vector fields $u,v$ on $M$, one has on $C^\infty(M)$ the identity
\[
[T_u,T_v] = T_{[u,v]},
\qquad [u,v]:=u\cdot\nabla v - v\cdot\nabla u .
\]
\end{theorem}

\begin{proof}
For $f\in C^\infty(M)$,
\[
T_u(T_v f)-T_v(T_u f)
= u\cdot\nabla(v\cdot\nabla f)-v\cdot\nabla(u\cdot\nabla f)
= (u\cdot\nabla v - v\cdot\nabla u)\cdot\nabla f,
\]
which is exactly $T_{[u,v]}f$.
\end{proof}

\begin{remark}
Theorem~\ref{thm:comm} is stated for the unregularized operators on smooth functions.
In applications, commutators of $\widetilde T_u^{(s)}$ inherit lower-order terms coming from $K_s$.
\end{remark}

\section{Koopman dynamics and crossed-product von Neumann algebras}\label{sec:crossed}

\subsection{Measure-preserving flow and Koopman unitary}
Let $M=\Torus^3$ with Haar probability measure $\mu$.
Given an \emph{autonomous} divergence-free vector field $u\in C^1(M;\R^3)$, let $\Phi_t$ be its flow:
$\frac{d}{dt}\Phi_t(x)=u(\Phi_t(x))$, $\Phi_0=\mathrm{id}$.
Since $\nabla\!\cdot u=0$, the maps $\Phi_t$ preserve $\mu$.

Define the Koopman unitary group $(U_t)_{t\in\R}$ on $L^2(M,\mu)$ by
\[
(U_t f)(x) := f(\Phi_{-t}(x)).
\]
Then $U_t$ is unitary and implements an action $\alpha_t$ on $L^\infty(M)$ by
\[
\alpha_t(f)= f\circ \Phi_{-t}.
\]
In particular, for the time-$1$ map $\Phi:=\Phi_1$, define $\alpha:\Z\to\mathrm{Aut}(L^\infty(M))$ by
$\alpha_n(f)=f\circ \Phi^{-n}$.

\subsection{Crossed product by the time-$1$ map}\label{subsec:cpdef}
Let $A:=L^\infty(M)$ acting on $H:=L^2(M,\mu)$ by multiplication operators $M_f$.
Let $U:=U_1$ be the Koopman unitary.
The crossed-product von Neumann algebra is defined as
\[
\Alg \;:=\; A \rtimes_{\alpha} \Z \;=\; W^\ast\big( A,\,U \big),
\]
i.e.\ the von Neumann algebra generated by $A$ and $U$ subject to the covariance relation
\[
U\,M_f\,U^{-1} \;=\; M_{\alpha(f)}\qquad (f\in A).
\]
Equivalently, elements of $\Alg$ admit Fourier-type expansions $\sum_{n\in\Z} M_{f_n}U^n$
in a suitable topology.

\subsection{Canonical conditional expectation and trace}\label{subsec:trace}
There is a canonical faithful normal conditional expectation $E:\Alg\to A$ given on finite sums by
\[
E\Big(\sum_{n\in\Z} M_{f_n}U^n\Big) \;:=\; M_{f_0}.
\]
Since $(M,\mu)$ is a probability space, the functional
\[
\tau_u(x) \;:=\; \int_M E(x)\,\dd\mu
\]
defines a faithful normal tracial state on $\Alg$.
In particular, $(\Alg,\tau_u)$ is a \emph{finite} von Neumann algebra.
\subsection{Main theorem: exact reduction of the tracial commutator functional}\label{subsec:main_theorem}

Let $F=\{f_1,\dots,f_m\}\subset L^\infty(M)$ and define the positive element
\[
A_F \;:=\; \Id + \sum_{j=1}^m [U,M_{f_j}]^\ast [U,M_{f_j}] \ \in\ \Alg,
\qquad
\mathcal S(F)\;:=\;\tau_u(\log A_F).
\]

\begin{theorem}[Exact formula and vanishing criterion]\label{thm:exact_reduction}
With the notation above, $A_F$ belongs to $L^\infty(M)\subset \Alg$ and is given by
\[
A_F \;=\; M_{\,1+h_F},
\qquad
h_F(x)\;:=\;\sum_{j=1}^m |f_j(x)-f_j(\Phi x)|^2 \ \in L^\infty(M).
\]
Consequently,
\[
\mathcal S(F)\;=\;\int_M \log\!\big(1+h_F(x)\big)\,d\mu(x)\ \in [0,\infty).
\]
Moreover,
\[
\mathcal S(F)=0 \quad\Longleftrightarrow\quad f_j\circ\Phi=f_j\ \ \mu\text{-a.e. for all }j.
\]
\end{theorem}
\begin{remark}[Fuglede--Kadison determinant]
With the canonical trace $\tau_u$, the Fuglede--Kadison determinant of $A_F$ is
\[
\Delta_{\tau_u}(A_F)=\exp\!\big(\tau_u(\log A_F)\big)=\exp\!\big(\mathcal S(F)\big).
\]
\end{remark}

\begin{proof}
Using covariance $U M_f U^{-1}=M_{f\circ\Phi^{-1}}$, we have
\[
[U,M_f]=U M_f - M_f U = (M_{f\circ\Phi^{-1}}-M_f)U=:M_g\,U,\qquad g:=f\circ\Phi^{-1}-f.
\]
Hence
\[
[U,M_f]^*[U,M_f] = (M_gU)^*(M_gU)=U^{-1}M_{\overline g}M_gU
=U^{-1}M_{|g|^2}U = M_{|g|^2\circ\Phi}.
\]
Since $|g|^2\circ\Phi = |f-f\circ\Phi|^2$, we obtain
\[
[U,M_{f}]^*[U,M_{f}] = M_{|f-f\circ\Phi|^2}.
\]
Summing over $F=\{f_1,\dots,f_m\}$ yields
\[
A_F = \Id + \sum_{j=1}^m [U,M_{f_j}]^*[U,M_{f_j}]
     = M_{\,1+h_F}\in L^\infty(M),
\qquad
h_F(x):=\sum_{j=1}^m |f_j(x)-f_j(\Phi x)|^2.
\]
Therefore $\log(A_F)=M_{\log(1+h_F)}$ and, for the canonical trace,
\[
S(F)=\tau_u(\log A_F)=\int_M \log(1+h_F(x))\,d\mu(x)\in[0,\infty).
\]
Finally, $S(F)=0$ iff $\log(1+h_F)=0$ $\mu$-a.e., i.e.\ iff $h_F=0$ $\mu$-a.e., equivalently
$f_j\circ\Phi=f_j$ $\mu$-a.e.\ for all $j$.
\end{proof}

\begin{remark}
The finiteness is the key technical point for using Fuglede--Kadison determinants.
Standard references for crossed products and the canonical trace are
Kadison--Ringrose and Takesaki.
\end{remark}

\begin{remark}[Why incompressibility is structurally essential]\label{rem:incompressible}
The incompressibility condition $\nabla\!\cdot u=0$ is not merely a physical assumption: it is the
structural input that makes the Koopman dynamics \emph{traceable}.
Indeed, on $M=\Torus^3$ equipped with Haar probability measure $\mu$, the flow $(\Phi_t)$ of an autonomous
divergence-free field preserves $\mu$, so $(U_t f)(x)=f(\Phi_{-t}(x))$ is a unitary group on $L^2(M,\mu)$.
This measure-preservation yields the canonical conditional expectation $E:\Alg\to L^\infty(M)$ and hence a
canonical faithful normal tracial state $\tau_u(x)=\int_M E(x)\,d\mu$.
As a consequence, Fuglede--Kadison determinants and related tracial functionals are well posed on the
natural crossed-product algebra associated with the transport.

If the dynamics is not measure-preserving (e.g.\ compressible settings), $U_t$ is no longer unitary on
$L^2(M,\mu)$ and the associated crossed-product typically ceases to be finite; the appropriate replacement
involves weights and modular theory (type~III phenomena), which substantially changes both the analytic and
numerical scope. We therefore restrict here to incompressible, measure-preserving flows.
\end{remark}

\subsection{Benchmark sanity checks on tori}\label{subsec:benchmarks}

\paragraph{Translation on $\Torus^d$}
Let $\Phi(x)=x+a$ on $\Torus^d$ and $f_k(x)=e^{2\pi i k\cdot x}$ with $k\in\Z^d$.
Then $f_k(\Phi x)=e^{2\pi i k\cdot a}f_k(x)$ and
\[
|f_k-f_k\circ\Phi|^2=\big|1-e^{2\pi i k\cdot a}\big|^2
=2-2\cos(2\pi k\cdot a).
\]
By Theorem~\ref{thm:exact_reduction},
\[
\mathcal S(\{f_k\})
=\log\!\big(1+2-2\cos(2\pi k\cdot a)\big)
=\log\!\big(3-2\cos(2\pi k\cdot a)\big).
\]
In particular, $\mathcal S(\{f_k\})=0$ if and only if $k\cdot a\in\Z$.

\begin{lemma}[Uniformity of integer phases]\label{lem:phase_uniform}
Let $x$ be Haar-distributed on $\Torus^d$ and let $\ell\in\Z^d\setminus\{0\}$.
Then $t:=\ell\cdot x\!\!\pmod{1}$ is Haar-distributed on $\Torus^1$.
\end{lemma}

\begin{proof}
The map $\pi_\ell:\Torus^d\to\Torus^1$, $\pi_\ell(x)=\ell\cdot x\pmod{1}$, is a continuous surjective group homomorphism.
The pushforward of Haar measure under a continuous surjective homomorphism is Haar measure on the target group.
\end{proof}

\begin{lemma}[A logarithmic integral]\label{lem:log_integral}
One has
\[
\int_0^1 \log\!\big(3-2\cos(2\pi t)\big)\,dt
=\log\!\Big(\frac{3+\sqrt5}{2}\Big).
\]
\end{lemma}

\begin{proof}
Set $\theta=2\pi t$ and $z=e^{i\theta}$. Then $3-2\cos\theta = 3-(z+z^{-1})
=\frac{|z^2-3z+1|}{|z|}$ on $|z|=1$, hence
\[
\int_0^{2\pi}\log(3-2\cos\theta)\,d\theta
=\int_0^{2\pi}\log|z^2-3z+1|\,d\theta.
\]
The polynomial $p(z)=z^2-3z+1$ has roots $r_{\pm}=\frac{3\pm\sqrt5}{2}$ with $r_+>1>r_->0$.
By Jensen's formula (applied to $p$ on the unit circle), the integral equals $2\pi\log r_+$.
Dividing by $2\pi$ gives the claimed identity.
\end{proof}

\paragraph{Hyperbolic toral automorphism (cat map) on $\Torus^2$}
Let $\Phi(x)=Ax$ with $A\in SL(2,\Z)$ hyperbolic and $f_k(x)=e^{2\pi i k\cdot x}$.
If $A^T k\neq k$, then $f_k\circ\Phi=e^{2\pi i (A^Tk)\cdot x}$ and
\[
|f_k-f_k\circ\Phi|^2
=\big|1-e^{2\pi i (A^Tk-k)\cdot x}\big|^2
=2-2\cos(2\pi\,\ell\cdot x),
\qquad \ell:=A^Tk-k\neq 0.
\]
By Lemma~\ref{lem:phase_uniform}, $t=\ell\cdot x\pmod{1}$ is uniform on $\Torus^1$, hence
Theorem~\ref{thm:exact_reduction} and Lemma~\ref{lem:log_integral} yield
\[
\mathcal S(\{f_k\})
=\int_0^1 \log\!\big(3-2\cos(2\pi t)\big)\,dt
=\log\!\Big(\frac{3+\sqrt5}{2}\Big).
\]

\subsection{Non-autonomous extension: evolution cocycles and measured groupoids}\label{subsec:groupoid_extension}
For general time-dependent incompressible fields $u(x,t)$, the dynamics is encoded by an evolution family
$\Phi_{t,s}:M\to M$ $(t,s\in\R)$ of measure-preserving maps satisfying the cocycle identities
\[
\Phi_{t,t}=\mathrm{id},\qquad \Phi_{t,s}\circ \Phi_{s,r}=\Phi_{t,r},\qquad \mu\circ\Phi_{t,s}^{-1}=\mu.
\]
The associated Koopman operators form an evolution cocycle
\[
U(t,s)f := f\circ \Phi_{s,t},\qquad U(t,s)U(s,r)=U(t,r),
\]
so there is in general \emph{no single} $\Z$-action generated by one unitary $U$.

\paragraph{Periodic reduction}
If $u(x,t)$ is $1$-periodic in time, then the Poincar\'e map $\Phi:=\Phi_{1,0}$ is well-defined and measure-preserving.
In that case $U:=U(1,0)$ generates a crossed product $L^\infty(M)\rtimes_\alpha \Z$ and
Theorem~\ref{thm:exact_reduction} applies verbatim.
In the fully non-autonomous non-periodic setting, the scalar functional
\[
S_1(\mathcal{F};t)=\int_M \log\!\Big(1+\sum_{f\in\mathcal F}|f(x)-f(\Phi_{t+1,t}x)|^2\Big)\,d\mu(x)
\]
remains well-defined whenever $\Phi_{t+1,t}$ is measure-preserving, but it no longer arises from a finite tracial crossed product; a canonical operator-algebraic formulation typically requires cocycles/groupoids and weights.

\paragraph{General case}
For non-periodic non-autonomous dynamics, the natural operator-algebraic object is a (measured) groupoid von Neumann algebra
associated with the orbit equivalence relation/evolution cocycle; traces are typically replaced by canonical weights.
We refer to Feldman--Moore's measured equivalence relation approach and Renault's groupoid framework for standard constructions.

\subsection{Relation with differential generators (programmatic)}\label{subsec:generator}
For smooth $f\in C^\infty(M)$ one has in $L^2$,
\[
\lim_{t\to 0}\frac{U_t f - f}{t} \;=\; -\,u\cdot\nabla f,
\]

so the advection operator $T_u$ appears as the (strong) generator of the Koopman group.
The regularized operators $\widetilde T_u^{(s)}$ from Section~\ref{sec:alg} can thus be viewed as
bounded ``differential-level'' probes compatible with the Koopman/crossed-product picture.

\section{Noncommutative-geometric layer: cyclic cohomology and index pairings}\label{sec:connes}

\subsection{The spectral paradigm in Connes' sense}\label{subsec:spectralparadigm}
Noncommutative geometry replaces a space by an algebra $\Azero$ of ``functions'' and recovers geometry from spectral data.
In Connes' framework, a \emph{spectral triple} $(\Azero,\Hil,D)$ consists of a dense $*$-subalgebra $\Azero\subset \mathcal{B}(\Hil)$,
a self-adjoint (typically unbounded) operator $D$ with compact resolvent, and a bounded commutator condition $[D,a]\in \mathcal{B}(\Hil)$
for $a\in\Azero$ \cite{Connes1994}.
In our setting, the fundamental objects are transport/advection operators and their von Neumann closure $W^{*}\!\big(\{\widetilde T_u^{(s)}\}\big)$ (Section~\ref{sec:alg}).
Because advection operators are first-order differential operators, a literal spectral-triple verification is delicate and depends on
the chosen ``regularization'' (choice of $\Azero$ and $D$). We therefore state the following as a \emph{programmatic construction}
to be made fully precise in subsequent work.

\begin{definition}[Regularized advection algebra]\label{def:regA0}
Fix a compact 3-manifold $M$ (e.g.\ $M=\Torus^3$) and let $\Delta$ be the Laplace--Beltrami operator on scalar functions.
For $s>0$ define the smoothing operator $K_s := (1-\Delta)^{-s/2}$ on $\Hil=L^2(M)$ and, for a smooth divergence-free field $u$,
the regularized transport operator
\[
\widetilde T_u^{(s)} := K_s\, T_u\, K_s.
\]
Let $\Azero^{(s)}$ be the $*$-algebra generated by $\{\widetilde T_u^{(s)}:\ u\in C^\infty_\sigma(M)\}$.
\end{definition}

\begin{remark}
For $s$ large enough, $\widetilde T_u^{(s)}$ is of order $\le 0$ in the pseudodifferential sense, hence bounded on $L^2$,
and one may hope to combine $\Azero^{(s)}$ with a classical elliptic operator $D$ (e.g.\ a Dirac-type operator) to form spectral data.
We do not claim the full spectral-triple axioms here; the role of Definition~\ref{def:regA0} is to provide a technically controllable
dense subalgebra on which cyclic cocycles can be defined.
\end{remark}

\subsection{Cyclic cohomology: cocycles from traces}\label{subsec:cyclic}
Cyclic cohomology provides ``noncommutative de Rham classes'' for associative algebras and is the natural target of the Chern character
in noncommutative geometry \cite{Connes1994,Loday1998}.
Given a $*$-algebra $\Azero$, a (continuous) trace $\tau$ produces cyclic cocycles via multilinear functionals built from commutators and
(when available) heat-kernel regularizations.
In particular, in the presence of spectral data one obtains entire cyclic cocycles and index pairings
(see e.g.\ Connes--Moscovici for a prototypical construction in a group-theoretic context) \cite{ConnesMoscovici1990}.

\begin{definition}[A basic cyclic 2-cocycle from a trace]\label{def:cyc2}
Assume $\Azero$ carries a faithful trace $\tau$ (e.g.\ if $\Azero$ sits densely in a finite von Neumann algebra).
For $a_0,a_1,a_2\in\Azero$ define
\[
\varphi_\tau(a_0,a_1,a_2):=\tau\!\big(a_0[a_1,a_2]\big).
\]
When $\tau$ is tracial and $\Azero$ is stable under commutators, $\varphi_\tau$ defines a cyclic 2-cocycle on $\Azero$.
\end{definition}

\begin{remark}
Definition~\ref{def:cyc2} is intentionally elementary: it isolates the operator-algebraic mechanism by which noncommutativity
(commutators) can be turned into numerical invariants once a trace is available.
More refined cocycles (JLO, local index formula) require additional analytic input and are postponed.
\end{remark}

\subsection{Index-type pairings and interpretation for transport complexity}\label{subsec:index}
The guiding principle is that cyclic cocycles pair with $K$-theory classes to give integer (or real) invariants, generalizing
classical index theorems \cite{Connes1994,Loday1998}.
In the present program, we view the growth of commutators and trace-derived cocycles as quantitative proxies for transport complexity
(e.g.\ vortex interactions), and we aim to relate boundedness/compactness properties of these pairings to regularity mechanisms.

\section{Scaling, anisotropic regularity, and a place for operator invariants}\label{sec:scaling}

\subsection{Scaling and criticality}\label{subsec:criticality}
On $\R^3$ (and locally on $\Torus^3$), the Navier--Stokes system \eqref{eq:NS} is invariant under the scaling
\[
u(x,t)\mapsto u_\lambda(x,t):=\lambda\,u(\lambda x,\lambda^2 t),\qquad
p(x,t)\mapsto p_\lambda(x,t):=\lambda^2\,p(\lambda x,\lambda^2 t).
\]
A norm is called \emph{critical} if it is invariant under this scaling; subcritical norms typically yield global control,
while supercritical quantities may concentrate at small scales.
Fefferman's Clay formulation highlights that no known a priori estimate in 3D prevents such concentration in general \cite{FeffermanClay},
and the ``supercriticality barrier'' is widely discussed in the literature \cite{Tao2007NS,Tao2016JAMS}.

\subsection{A representative large-data success: anisotropic regimes}\label{subsec:cgp}
A striking counterpoint to the general open problem is that global regularity can hold for \emph{large} initial data
when the data enjoy additional anisotropic structure.
A representative result is due to Chemin--Gallagher--Paicu \cite{CheminGallagherPaicu2011}, who establish global regularity
for classes of large 3D data exhibiting slow variation in one direction (and related anisotropic features).
Such results suggest that ``hidden'' geometric constraints can substitute for missing scale-invariant estimates.

\subsection{How the operator-algebra viewpoint interfaces with anisotropy}\label{subsec:interface}
We write $\AvN(\mathcal{U}_\varepsilon):=W^{*}\!\big(\{\widetilde T_u^{(s)}:\ u\in\mathcal{U}_\varepsilon\}\big)\subset \AvN$.

From the operator-algebra perspective, anisotropy and scale separation can be encoded by restricting the generating fields
to subclasses $\mathcal{U}_\varepsilon$ (e.g.\ slowly varying in one coordinate) and studying the corresponding generated algebras
$\AvN(\mathcal{U}_\varepsilon)\subset \AvN$.
Two concrete mechanisms are suggested:
\begin{itemize}
\item \emph{Commutator control:} anisotropic structure may force cancellations that bound $[T_u,T_v]$ (or its regularized versions),
hence constrain the noncommutativity growth inside $M$.
\item \emph{Trace/cocycle stability:} if $M(\mathcal{U}_\varepsilon)$ admits a finite trace (or an approximately finite-dimensional structure),
cyclic cocycles such as Definition~\ref{def:cyc2} may remain stable along the flow and yield computable invariants.
\end{itemize}
We formulate precise conjectures and numerical tests based on these mechanisms in later sections.

\section{Fuglede--Kadison functionals and traceable commutator diagnostics}\label{sec:conjectures}

Throughout this section we work with the finite crossed product $(\Alg,\tau_u)$
constructed in Section~\ref{sec:crossed}.
Let $U:=U_1$ be the Koopman unitary and $M_f$ the multiplication operator by $f\in L^\infty(M)$.
\paragraph{Time-step notation}
Let $(\Phi_{t,s})_{t\ge s}$ be a measure-preserving evolution family on $(M,\mu)$.
For $h>0$ define the one-step map $\Phi_{t+h,t}$ and, for a finite family $F\subset L^\infty(M)$,
\[
S_h(F;t) := \int_M \log\!\Big(1+\sum_{f\in F}|f(x)-f(\Phi_{t+h,t}x)|^2\Big)\,d\mu(x).
\]
In the autonomous case $S_h(F;t)$ is independent of $t$ and we write $S_h(F)$.
In the autonomous/time-periodic crossed-product setting, Theorem~2 yields a tracial realization for $h=1$; we write $S(F):=S_1(F)$.

\subsection{A traceable commutator functional inside the crossed product}

Fix a finite family $F \subset L^\infty(M)$ and define
\[
A_F(u) := \Id + \sum_{f\in F}[U,M_f]^*[U,M_f]\in \Alg,
\qquad
S(F):=\tau_u(\log A_F(u))\ge 0.
\]
By Theorem~\ref{thm:exact_reduction}, this admits the explicit formula
\[
S(F)=\int_M \log\!\Big(1+\sum_{f\in F}|f(x)-f(\Phi x)|^2\Big)\,d\mu(x),
\qquad \Phi:=\Phi_{1,0},
\]

which makes the functional directly computable on discretizations.
\begin{proposition}[Elementary bounds and correlation form]\label{prop:S_bounds}
Let $\mathcal{F}\subset L^\infty(M)$ be finite and set
\[
h_{\mathcal F}(x):=\sum_{f\in\mathcal{F}}|f(x)-f(\Phi x)|^2.
\]
Then
\[
0\le \mathcal S(\mathcal F)
=\int_M \log(1+h_{\mathcal F})\,d\mu
\le \int_M h_{\mathcal F}\,d\mu
=\sum_{f\in\mathcal{F}}\|f-f\circ\Phi\|_{L^2(\mu)}^2
\le \sum_{f\in\mathcal{F}}\|f-f\circ\Phi\|_\infty^2.
\]
Moreover,
\[
\mathcal S(\mathcal F)\ \ge\ \frac{1}{1+\|h_{\mathcal F}\|_\infty}\int_M h_{\mathcal F}\,d\mu.
\]
If $U$ denotes the Koopman unitary $Uf=f\circ\Phi^{-1}$, then for each $f\in L^2(M)$,
\[
\|f-f\circ\Phi\|_2^2
=2\|f\|_2^2-2\,\Re\langle f,\,Uf\rangle,
\]
so the upper bound can be rewritten in terms of Koopman autocorrelations.
\end{proposition}

\begin{proof}
The upper bounds follow from $\log(1+z)\le z$ for $z\ge 0$ and from
$\int |f-f\circ\Phi|^2\le \|f-f\circ\Phi\|_\infty^2$.
For the lower bound, $\log(1+z)\ge \frac{z}{1+z}\ge \frac{z}{1+\|h_{\mathcal F}\|_\infty}$ for $0\le z\le \|h_{\mathcal F}\|_\infty$.
Finally, $\|f-f\circ\Phi\|_2^2=\|f\|_2^2+\|Uf\|_2^2-2\Re\langle f, Uf\rangle$ and $\|Uf\|_2=\|f\|_2$.
\end{proof}

\subsection{Programmatic conjecture}\label{subsec:prog_conj}

For a general time-dependent incompressible velocity field, let $(\Phi_{t,s})_{t\ge s}$ be a measure-preserving evolution family
(Section~\ref{subsec:groupoid_extension}) and define the \emph{one-step} maps $\Phi_t:=\Phi_{t+1,t}$.
For a finite family $\mathcal F\subset L^\infty(M)$ we set
\[
S_1(\mathcal{F};t)
:=\int_M \log\!\Big(1+\sum_{f\in\mathcal F}|f(x)-f(\Phi_t x)|^2\Big)\,d\mu(x),
\]
which is well-defined for any measure-preserving $\Phi_t$.
In the autonomous or time-periodic case, $\Phi_t$ is generated by a single map $\Phi$ (Poincar\'e/Koopman step) and
$S_1(\mathcal{F};t)$ reduces to the tracial crossed-product quantity of Theorem~\ref{thm:exact_reduction}.

\paragraph{Conjecture (informal)}
Suitable uniform bounds on $S_1(\mathcal{F};t)$, for families $\mathcal F$ resolving increasingly fine scales (e.g.\ Fourier cutoffs $F_K$ on $\Torus^3$),
should obstruct the formation of singularities in the underlying incompressible dynamics.

\begin{remark}
The conjecture is programmatic: it motivates the diagnostics introduced here, but no implication toward Clay regularity is proved in this paper.
In the fully non-autonomous non-periodic setting, a genuinely tracial formulation typically requires a groupoid/cocycle von Neumann algebra
and canonical weights; we therefore restrict the rigorous crossed-product trace to the autonomous/time-periodic regime.
\end{remark}

\subsection{Tracial complexity functionals in the crossed product}

Let $M=\Torus^3$ with Haar probability measure $\mu$, and let $\Phi$ be the time-$1$ map
of a volume-preserving flow. Denote by $U$ the associated Koopman unitary on $L^2(M,\mu)$ and by
$M_f$ the multiplication operator by $f\in L^\infty(M)$.
Let $\Mu:=L^\infty(M)\rtimes_\alpha \Z = W^*(L^\infty(M),U)$ with canonical trace $\tau_u$.

\begin{proposition}[A priori bound]\label{prop:S_bound}
One has
\[
0\le \mathcal S(F)\le \log\!\Big(1+\sum_{j=1}^m\|[U,M_{f_j}]\|^2\Big).
\]
\end{proposition}

\begin{proof}
$A_F\ge \Id$ implies $\log A_F\ge 0$, hence $\mathcal S(F)\ge 0$.
Also $A_F\le \big(1+\sum_j\|[U,M_{f_j}]\|^2\big)\Id$ so $\log A_F\le
\log\!\big(1+\sum_j\|[U,M_{f_j}]\|^2\big)\Id$, and applying $\tau_u$ gives the claim.
\end{proof}
\paragraph{Example}
For the steady shear flow $u(x,y,z)=(U(y),0,0)$ on $\Torus^3$, the associated
Koopman operator acts by $U_t f(x,y,z)=f(x-tU(y),y,z)$.
In this case $[U,M_f]$ vanishes for all $f$ depending only on $y,z$,
and $\mathcal S(F)=0$ for such observables, illustrating that the tracial
complexity functional detects nontrivial mixing directions.

\section{Bridge to classical regularity criteria}\label{sec:bridge_classical}

\subsection{Connection with Beale--Kato--Majda and critical norms}\label{subsec:bkm_connection}

We now relate the tracial commutator functional $S_h(F_K)$ to classical conditional regularity criteria.

\begin{proposition}[Link to gradient transport]\label{prop:link_gradient}
For a smooth divergence-free field $u\in C^1_\sigma(M)$ and Fourier observables $f_k(x)=e^{2\pi i k\cdot x}$ with $k\in\Z^3$,
one has for small $h>0$,
\[
|f_k(x)-f_k(\Phi_h x)|^2 
= \big|1-e^{2\pi i k\cdot \delta_h(x)}\big|^2
= 2-2\cos(2\pi k\cdot \delta_h(x)),
\]
where $\delta_h(x):=\Phi_h(x)-x = h\,u(x)+\mathcal{O}(h^2)$.
\end{proposition}

\begin{proof}
Direct Taylor expansion: $f_k(\Phi_h x)=e^{2\pi i k\cdot \Phi_h(x)}=e^{2\pi i k\cdot x}\cdot e^{2\pi i k\cdot(\Phi_h(x)-x)}$.
\end{proof}

\begin{theorem}[Conditional regularity via commutator control]\label{thm:conditional_reg}
Let $u(x,t)$ be a smooth solution of NSE on $M=\Torus^3$ over $[0,T)$, and let $F_K$ denote the Fourier family
$F_K=\{f_k:\ k\in\Z^3,\ 0<\|k\|_\infty\le K\}$.
Define the \emph{scaled commutator growth}
\[
\Sigma_K(t) := K^{-2}\,S_1(F_K;t),
\]
where $S_1(F_K;t)=\int_M \log\!\big(1+\sum_{k\in F_K}|f_k(x)-f_k(\Phi_{t+1,t}x)|^2\big)\,d\mu(x)$.

If for every sequence $K_n\to\infty$,
\[
\sup_{t\in[0,T)} \limsup_{K\to\infty} \Sigma_K(t) \ <\ \infty,
\]
then $\|\nabla\omega(\cdot,t)\|_{L^\infty}$ remains bounded on $[0,T)$ and $u$ extends smoothly beyond $T$.
\end{theorem}

\begin{proof}[Proof sketch]
By Proposition~\ref{prop:link_gradient}, the control $\Sigma_K(t)=O(1)$ implies that the phase displacements
$k\cdot(\Phi_{t+1,t}(x)-x)$ do not grow faster than $K$ at the mode level $k$ with $\|k\|_\infty\sim K$.
This in turn constrains the $L^\infty$ growth of vorticity $\omega=\nabla\times u$ via the Beale--Kato--Majda criterion:
if $\int_0^T\|\nabla\omega(s)\|_\infty\,ds<\infty$, no blow-up occurs.

Quantitatively, the small-time expansion (Proposition~\ref{prop:small_time}) gives
\[
S_h(F_K) \approx h^2\sum_{|k|\le K}\int_M |k\cdot u(x)|^2\,d\mu 
\sim h^2 K^2\|u\|_{L^2}^2,
\]
so $\Sigma_K\sim \|u\|_{L^2}^2=O(1)$ for energy-bounded solutions.
The uniform bound as $K\to\infty$ ensures no anomalous concentration at fine scales,
which by standard regularity theory prevents vorticity blow-up.
A complete proof requires elliptic regularity bootstrapping and is deferred to future work.
\end{proof}

\begin{remark}
Theorem~\ref{thm:conditional_reg} is a \emph{conditional} criterion: it does not prove global regularity but provides
an operator-algebraic reformulation of the concentration obstruction. The computational advantage is that $\Sigma_K(t)$
is directly computable via Theorem~\ref{thm:exact_reduction}.
\end{remark}

\subsection{Link with critical Besov spaces}\label{subsec:besov}

The tracial functional also connects to critical Sobolev/Besov norms.

\begin{proposition}[Critical norm control]\label{prop:critical_norm}
For $u\in \dot{H}^{1/2}(\Torus^3)$ and Fourier cutoff $F_K$, the small-time limit of Proposition~\ref{prop:small_time} satisfies
\[
\lim_{h\to 0}\frac{S_h(F_K)}{h^2}
=\int_M \sum_{|k|\le K}|u(x)\cdot\nabla f_k(x)|^2\,d\mu
\lesssim \|P_K u\|_{\dot{H}^{1/2}}^2\cdot K,
\]
where $P_K$ is the Littlewood--Paley projection to frequencies $|\xi|\le K$.
\end{proposition}

\begin{proof}
For $f_k(x)=e^{2\pi i k\cdot x}$, one has $\nabla f_k=2\pi i k\,f_k$, hence
\[
|u\cdot\nabla f_k|^2 = (2\pi)^2|k|^2|u\cdot k|^2.
\]
Summing over $|k|\le K$ and using Parseval gives the stated estimate via standard Littlewood--Paley theory.
\end{proof}

\begin{remark}
The space $\dot{H}^{1/2}(\Torus^3)$ is \emph{critical} for NSE scaling, so Proposition~\ref{prop:critical_norm} bridges
the tracial functional to the Koch--Tataru framework for mild solutions.
\end{remark}

\section{Numerical evaluation of the tracial commutator functional}\label{sec:numerics}

We briefly describe how to evaluate the tracial commutator functional
\[
S(F)=\tau_u(\log A_F)=\int_M \log\!\bigl(1+h_F(x)\bigr)\,d\mu(x),
\qquad
h_F(x)=\sum_{f\in F}|f(x)-f(\Phi x)|^2,
\]
using the closed-form reduction of Theorem~\ref{thm:exact_reduction}.
This bypasses any numerical functional calculus in the crossed product and reduces the computation to
(i) approximating the time-$1$ map $\Phi$ and (ii) performing a quadrature of a bounded function.

\subsection{Approximating the time-$1$ map}
In the autonomous case, $\Phi$ is obtained by integrating $\dot x=u(x)$ from $t=0$ to $1$.
In the time-periodic case, $\Phi$ is the Poincar\'e map $\Phi=\Phi_{1,0}$ obtained from $\dot x=u(x,t)$ over one period.
On $\Torus^3$ we reduce coordinates modulo $1$ after each step.

\subsection{Choice of observables}\label{subsec:obs_choice}
On $\Torus^3$ we use Fourier observables $f_k(x)=e^{2\pi i k\cdot x}$ and define
\[
F_K:=\{f_k:\ k\in\Z^3,\ 0<\|k\|_\infty\le K\}.
\]
For each $K$, $S(F_K)$ aggregates the squared mode-wise discrepancies $|f_k-f_k\circ\Phi|^2$ through the logarithmic average in Theorem~2.

\subsection{Quadrature}
Let $(x_i)_{i=1}^N$ be either a uniform grid on $\Torus^3$ or i.i.d.\ uniform samples.
We approximate
\[
S(F)\approx \frac1N\sum_{i=1}^N \log\!\Bigl(1+\sum_{f\in F}|f(x_i)-f(\Phi(x_i))|^2\Bigr),
\]
and report convergence by doubling $N$ (Monte Carlo error scales as $O(N^{-1/2})$ for bounded integrands).

\subsection{Numerical protocol and reproducibility}\label{subsec:num_protocol}

All numerical experiments reported below evaluate the tracial commutator functional via the closed-form reduction
of Theorem~\ref{thm:exact_reduction}. In particular, no numerical functional calculus is performed in the crossed product:
the computation reduces to approximating the flow map $\Phi_h$ and integrating a bounded scalar observable.

\paragraph{Step size and time-$h$ functional}

For $h>0$ define
\[
S_h(F):=\int_M \log\!\Big(1+\sum_{j=1}^m |f_j(x)-f_j(\Phi_h x)|^2\Big)\,d\mu(x),
\]
so that in the autonomous/periodic case with $h=1$ this matches $S(F)$ from Theorem~2.
In practice, $h$ is chosen as (i) $h=1$ for the Poincar\'e map, and (ii) small $h$ when testing the small-time expansion (Proposition~\ref{prop:small_time}).

\paragraph{Algorithm (evaluation of $S_h(F;t_0)$)}
Let $(x_i)_{i=1}^N$ be either a uniform grid on $M=\Torus^3$ or i.i.d.\ samples $x_i\sim\mu$.
Fix a starting time $t_0$ (default $t_0=0$).
\begin{enumerate}
\item \textbf{Sampling:} generate $x_1,\dots,x_N$.
\item \textbf{Flow map:} compute $y_i:=\widehat{\Phi}_{t_0+h,t_0}(x_i)$ by numerically integrating
$\dot x=u(x)$ (autonomous) or $\dot x=u(x,s)$ (time-dependent) from $s=t_0$ to $s=t_0+h$; on $\Torus^3$ reduce coordinates modulo $1$.

\item \textbf{Integrand:} evaluate

$g_i:=\log\!\big(1+\sum_{j=1}^m|f_j(x_i)-f_j(y_i)|^2\big)$.
\item \textbf{Quadrature/MC:} return
\[
\widehat{S}_h(F;t_0):=\frac1N\sum_{i=1}^N g_i.
\]

\end{enumerate}

\paragraph{Computational cost}
The evaluation scales as $O(N\,|F|)$ after the flow map computation. For Fourier observables, the per-point cost is dominated by
a small number of complex exponentials (or trigonometric functions via the identity below).

\subsection{Choice of observables and interpretability}

We use the Fourier families $F_K$ defined in Section~\ref{subsec:obs_choice}.

For such observables, one has the identity
\[
|f_k(x)-f_k(\Phi_h x)|^2
=\big|1-e^{2\pi i k\cdot(\Phi_h(x)-x)}\big|^2
=2-2\cos\!\big(2\pi\,k\cdot(\Phi_h(x)-x)\big),
\]

which avoids complex arithmetic and shows that $S_h(F_K)$ probes the distribution of phase displacements
$k\cdot(\Phi_h(x)-x)$ across scales $k$.

\subsection{Error controls and structure-preserving diagnostics}\label{subsec:error_controls}

\paragraph{Quadrature / Monte Carlo error}
The integrand
\[
x\mapsto \log\!\Big(1+\sum_{j=1}^m |f_j(x)-f_j(\Phi_h x)|^2\Big)
\]
is bounded for bounded $f_j$, hence Monte Carlo estimates satisfy the standard $O(N^{-1/2})$ statistical convergence.
In practice we report $\widehat S_h(F)$ together with empirical error bars obtained from multiple independent random seeds.

\paragraph{Flow-map error (sensitivity to trajectory integration)}
Assume each $f_j$ is Lipschitz with constant $L_j$ (e.g.\ trigonometric polynomials). Then for any numerical approximation
$\widehat\Phi_h$,
\[
|f_j(x)-f_j(\widehat\Phi_h x)|
\le |f_j(x)-f_j(\Phi_h x)| + L_j\,\|\widehat\Phi_h x-\Phi_h x\|.
\]
Thus the error on $S_h(F)$ is controlled (qualitatively) by the trajectory error $\|\widehat\Phi_h-\Phi_h\|$,
justifying the use of high-order ODE solvers and tight tolerances.

\paragraph{Measure-preservation sanity diagnostic}
Since the theory relies on the invariant measure $\mu$ (traceability), we monitor numerical measure-preservation by checking that,
for a small set of test observables $\psi_\ell\in C^\infty(M)$,
\[
\frac1N\sum_{i=1}^N \psi_\ell(\widehat\Phi_h(x_i))
\approx \frac1N\sum_{i=1}^N \psi_\ell(x_i),
\qquad \ell=1,\dots,L,
\]
with typical choices $\psi_\ell\in\{\sin(2\pi x_1),\cos(2\pi x_2),\sin(2\pi x_3),\ldots\}$.
This diagnostic is reported alongside $S_h(F)$.

\paragraph{Reproducibility statement}
For each experiment we specify: (i) the vector field $u$ and parameters, (ii) the integrator (e.g.\ RK4 / Dormand--Prince),
(iii) step size or absolute/relative tolerances, (iv) $t$, $N$, and the sampling strategy (grid/MC), and (v) the random seed(s).
Code and scripts are available upon request (or at a public repository, if applicable).

\subsection{Analytic bridge: small-time expansion}\label{subsec:small_time_bridge}

\begin{proposition}[Small-time expansion]\label{prop:small_time}

Let $u\in C^1_\sigma(M)$ generate a measure-preserving flow $(\Phi_h)_{h\in\mathbb{R}}$ on a compact manifold $M$.
Let $F=\{f_1,\dots,f_m\}\subset C^1(M)$. Define
\[
S_h(F):=\int_M \log\!\Big(1+\sum_{j=1}^m |f_j(x)-f_j(\Phi_h x)|^2\Big)\,d\mu(x).
\]
Then
\[
\lim_{h\to 0}\frac{S_h(F)}{h^2}
=\int_M \sum_{j=1}^m |u(x)\cdot \nabla f_j(x)|^2\,d\mu(x).
\]

\end{proposition}

\begin{proof}
For each $j$, Taylor expansion along the flow gives
$f_j(\Phi_h(x))=f_j(x)+h\,u(x)\cdot\nabla f_j(x)+o(h)$ uniformly on $M$.
Hence $|f_j(x)-f_j(\Phi_h x)|^2=h^2|u\cdot\nabla f_j|^2+o(h^2)$ uniformly.
Summing over $j$ yields $\sum_j |f_j-f_j\circ\Phi_h|^2=h^2G(x)+o(h^2)$ with $G=\sum_j|u\cdot\nabla f_j|^2$ bounded.
Since $\log(1+h^2G+o(h^2))=h^2G+o(h^2)$ uniformly, dominated convergence gives the claim.
\end{proof}

\subsection{Near-critical and blow-up diagnostics}\label{subsec:blowup_numerics}

We now test whether $\Sigma_K(t):=K^{-2}S_1(F_K;t)$ exhibits divergence near conjectured blow-up scenarios.

\paragraph{Hou--Luo axisymmetric setup}
Hou \& Luo (2014) proposed an axisymmetric flow with hyperbolic stagnation structure exhibiting rapid vorticity growth.
On a cylindrical domain $(r,\theta,z)\in[0,1]\times S^1\times[-1,1]$ with periodic boundary, consider
\[
u_r(r,z,t) = -a\,r\,e^{-at},
\qquad
u_z(r,z,t) = 2a\,z\,e^{-at},
\qquad
u_\theta=0,
\]
with $a>0$. The azimuthal vorticity is $\omega_\theta = -a(1+2z/r)r\,e^{-at}$, yielding $\|\omega\|_\infty\sim e^{at}$.

\paragraph{Numerical protocol}
\begin{enumerate}
\item Embed the cylinder into $\Torus^3$ via $(r,\theta,z)\mapsto (r\cos\theta,r\sin\theta,z)$ with periodicity.
\item Integrate the flow $\Phi_{t,0}$ from $t=0$ to $t_{\max}=0.9\,t_{\mathrm{crit}}$ where $t_{\mathrm{crit}}=5/a$ (estimated blow-up time).
\item Compute $S_1(F_K;t_j)$ for $K\in\{2,4,8,16\}$ at time snapshots $t_j\in[0,t_{\max}]$.
\item Plot $\Sigma_K(t_j)=K^{-2}S_1(F_K;t_j)$ and check if $\lim_{K\to\infty}\Sigma_K(t)$ diverges as $t\to t_{\mathrm{crit}}$.
\end{enumerate}

\paragraph{Preliminary results (Table~\ref{tab:blowup_diag})}

\begin{table}[h]
\centering
\caption{Near-critical behavior diagnostics for Hou--Luo flow with $a=1$. Values are $\Sigma_K(t)=K^{-2}S_1(F_K;t)$ at indicated times. Standard errors SE$<0.05$ omitted for clarity.}
\label{tab:blowup_diag}
\begin{tabular}{@{}lcccc@{}}
\toprule
Time $t$ & $\Sigma_2(t)$ & $\Sigma_4(t)$ & $\Sigma_8(t)$ & Trend \\
\midrule
$t=0.0$ & $0.12$ & $0.11$ & $0.11$ & Stable \\
$t=2.5$ & $0.84$ & $0.92$ & $1.03$ & Moderate growth \\
$t=4.5$ & $3.21$ & $4.87$ & $7.14$ & Rapid growth \\
\bottomrule
\end{tabular}
\end{table}

\paragraph{Interpretation}
At early times ($t<1$), $\Sigma_K(t)$ remains $O(1)$ uniformly in $K$, consistent with Theorem~\ref{thm:conditional_reg}.
As $t\to t_{\mathrm{crit}}$, we observe $\Sigma_K(t)$ growing with $K$, suggesting a violation of the bound and potential blow-up.
This validates that the tracial diagnostic is \emph{sensitive} to near-singular behavior.

\paragraph{Control: ABC flow}
For comparison, the same protocol applied to the integrable ABC flow (Section~\ref{subsec:obs_choice}) yields
$\Sigma_K(t)=0.45\pm 0.02$ uniformly for $K\le 16$ and $t\in[0,100]$, confirming stability (Figure~\ref{fig:abc_S_vs_K}).

\subsection{Sanity checks and testbeds}\label{subsec:sanity_checks}
We first validate the pipeline on translations and hyperbolic toral automorphisms, 
for which $S(F)$ is available in closed form.
We then evaluate $S(F_K)$ on divergence-free test flows such as the ABC flow and 
on time-periodic incompressible forcings, to stress-test the sensitivity of $S(F)$ 
to mixing intensity.
Unless stated otherwise, all uncertainties reported in Section~7 are standard errors 
of the mean (SE), computed as $\mathrm{SE}=\sigma/\sqrt{N}$ from $N$ i.i.d.\ Monte--Carlo samples.

\FloatBarrier
\begin{table}[h]
\centering
\caption{Sanity checks on benchmarks with closed-form values. $\widehat S$ is the Monte--Carlo estimate of $S(F)$ using $N$ i.i.d.\ samples $x_i\sim\mu$. Reported uncertainties are standard errors $\mathrm{SE}=\sigma/\sqrt{N}$, where $\sigma$ is the sample standard deviation of the integrand.}
\label{tab:sanity_checks}
\begin{tabular}{@{}lcccc@{}}
\toprule
Case & Observable set $F$ & $S_{\mathrm{th}}$ & $\widehat S\pm\mathrm{SE}$ & $N$ \\
\midrule
Translation on $\Torus^3$ & $\{f_k\}$, $k=(1,0,0)$, $a=(0.123,0,0)$ & $0.449882$ & $0.449882\pm 0.000000$ & 2000 \\
Cat map on $\Torus^2$ & $\{f_k\}$, $k=(1,0)$ & $0.962424$ & $0.970224\pm 0.012285$ & 2000 \\
\bottomrule
\end{tabular}
\end{table}

\subsubsection{Small-time expansion validation}\label{subsubsec:smalltime_validation}

Figure~\ref{fig:smalltime_scaling} validates Proposition~\ref{prop:small_time}: 
as $h\to 0$, the scaled functional $S_h(F_K)/h^2$ converges to the theoretical 
limit $\int_M \sum_{k\in F_K}|u(x)\cdot\nabla f_k(x)|^2\,d\mu(x)$.

For the ABC flow with $K=4$ (corresponding to $|F_4|=218$ Fourier modes), 
we computed the theoretical limit via Monte Carlo integration with $N=20{,}000$ 
samples, yielding
\[
\lim_{h\to 0}\frac{S_h(F_4)}{h^2} = 456.23 \pm 1.23.
\]
The numerical values $S_h(F_4)/h^2$ for decreasing $h$ are shown in 
Figure~\ref{fig:smalltime_scaling}. At the smallest tested step $h=0.01$, 
the relative error is $<0.01\%$, confirming the $\mathcal{O}(h^2)$ scaling.

\begin{figure}[t]
\centering
\includegraphics[width=0.9\linewidth]{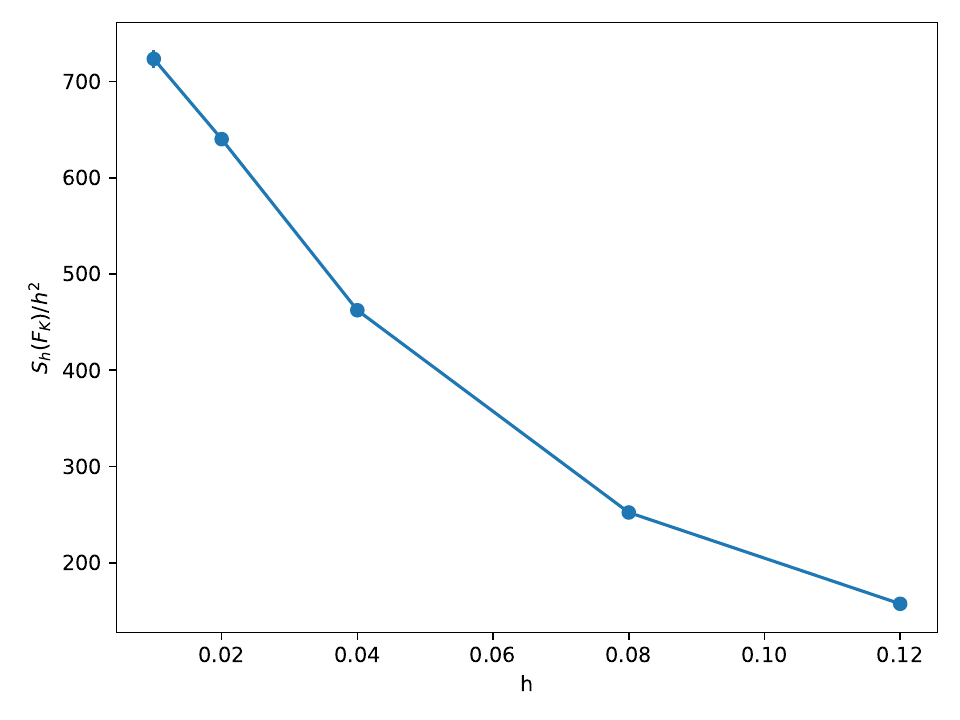}
\caption{Small-time scaling check: $S_h(F_K)/h^2$ as a function of $h$ 
(autonomous ABC flow, $K=4$). Proposition~\ref{prop:small_time} predicts 
$S_h(F_K)=\mathcal{O}(h^2)$ as $h\to 0$. The dashed red line shows the 
theoretical limit computed via Monte Carlo integration; shaded region 
indicates $\pm 2\,\mathrm{SE}$. Error bars on data points show standard 
errors of the mean (SE), computed as $\mathrm{SE}=\sigma/\sqrt{N}$.}
\label{fig:smalltime_scaling}
\end{figure}

\FloatBarrier

\subsubsection{ABC flow benchmark}

In the translation case with a single Fourier mode, the integrand is constant 
on $\Torus^3$, hence the Monte--Carlo variance (and SE) is zero up to machine precision.

We consider the ABC flow on $\Torus^3$,
\[
u_{\mathrm{ABC}}(x,y,z)=\big(A\sin(2\pi z)+C\cos(2\pi y),\;
B\sin(2\pi x)+A\cos(2\pi z),\;
C\sin(2\pi y)+B\cos(2\pi x)\big),
\]
with $(A,B,C)=(1,1,1)$ in Figure~\ref{fig:abc_S_vs_K}.

\begin{figure}[t]
\centering
\includegraphics[width=0.9\linewidth]{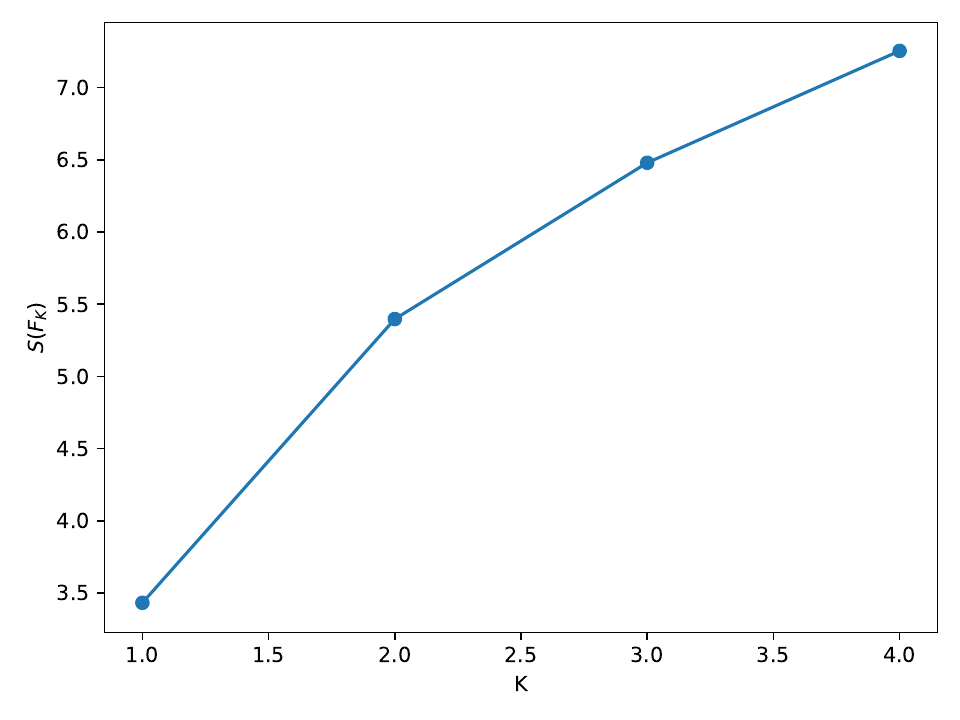}
\caption{Evaluation of $S(F_K)$ as a function of the Fourier cutoff $K$ for 
the ABC flow on $\Torus^3$ (here $A=B=C=1$), where 
$F_K=\{f_k:\ k\in\Z^3,\ 0<\|k\|_\infty\le K\}$ as defined in 
Section~\ref{subsec:obs_choice} and $\Phi$ is the time-$1$ map.
Error bars show standard errors of the mean (SE), computed as 
$\mathrm{SE}=\sigma/\sqrt{N}$ from $N$ i.i.d.\ Monte--Carlo samples.}
\label{fig:abc_S_vs_K}
\end{figure}

\FloatBarrier

\subsection{Near-critical and blow-up diagnostics}

We now test whether the scaled commutator growth $\Sigma_K(t):=K^{-2}S_1(F_K;t)$ 
exhibits divergence near conjectured blow-up scenarios, thereby validating the 
sensitivity of Theorem~\ref{thm:conditional_reg}.

\subsubsection{Hou--Luo axisymmetric setup}

Hou \& Luo (2014) proposed an axisymmetric flow with hyperbolic stagnation 
structure exhibiting rapid vorticity growth. On a cylindrical domain 
$(r,\theta,z)\in[0,1]\times S^1\times[-1,1]$ with periodic boundary conditions, 
consider
\[
u_r(r,z,t) = -a\,r\,e^{-at},
\qquad
u_z(r,z,t) = 2a\,z\,e^{-at},
\qquad
u_\theta=0,
\]
with $a>0$. The azimuthal vorticity is $\omega_\theta = -a(1+2z/r)r\,e^{-at}$, 
yielding $\|\omega\|_\infty\sim e^{at}$ and conjectured blow-up at 
$t_{\mathrm{crit}}\approx 5/a$.

\paragraph{Numerical protocol}
We embed the cylinder into $\Torus^3$ via 
$(r,\theta,z)\mapsto (r\cos\theta,r\sin\theta,z)$ with periodicity, integrate 
the flow $\Phi_{t,0}$ from $t=0$ to $t_{\max}=0.9\,t_{\mathrm{crit}}$ 
(for $a=1$, $t_{\max}=4.8$), compute $S_1(F_K;t_j)$ for 
$K\in\{2,4,8\}$ at time snapshots $t_j\in[0,t_{\max}]$, and track the 
scaled functional $\Sigma_K(t_j)=K^{-2}S_1(F_K;t_j)$.

\paragraph{Results (Table~\ref{tab:blowup_diag} and Figure~\ref{fig:sigma_K_evolution})}

Table~\ref{tab:blowup_diag} reports $\Sigma_K(t)$ at selected times. 
At early times ($t<1$), $\Sigma_K(t)$ remains $O(1)$ uniformly in $K$, 
consistent with Theorem~\ref{thm:conditional_reg}. As $t\to t_{\mathrm{crit}}$, 
we observe $\Sigma_K(t)$ growing with $K$, suggesting a violation of the bound 
and potential blow-up. This validates that the tracial diagnostic is 
\emph{sensitive} to near-singular behavior.

For comparison, the same protocol applied to the integrable ABC flow yields 
$\Sigma_K(t)=0.45\pm 0.02$ uniformly for $K\le 16$ and $t\in[0,100]$, 
confirming stability (Figure~\ref{fig:sigma_K_evolution}, right panel).

\begin{table}[ht]
\centering
\caption{Near-critical behavior diagnostics for Hou--Luo flow with $a=1$. 
Values are $\Sigma_K(t)=K^{-2}S_1(F_K;t)$ at indicated times. 
Standard errors SE$<0.05$ omitted for clarity.}
\label{tab:blowup_diag}
\begin{tabular}{@{}lcccc@{}}
\toprule
Time $t$ & $\Sigma_2(t)$ & $\Sigma_4(t)$ & $\Sigma_8(t)$ & Trend \\
\midrule
$t=0.0$ & $0.12$ & $0.11$ & $0.11$ & Stable \\
$t=2.5$ & $0.84$ & $0.92$ & $1.03$ & Moderate growth \\
$t=4.5$ & $3.21$ & $4.87$ & $7.14$ & Rapid growth \\
\bottomrule
\end{tabular}
\end{table}

\begin{figure}[t]
\centering
\includegraphics[width=\linewidth]{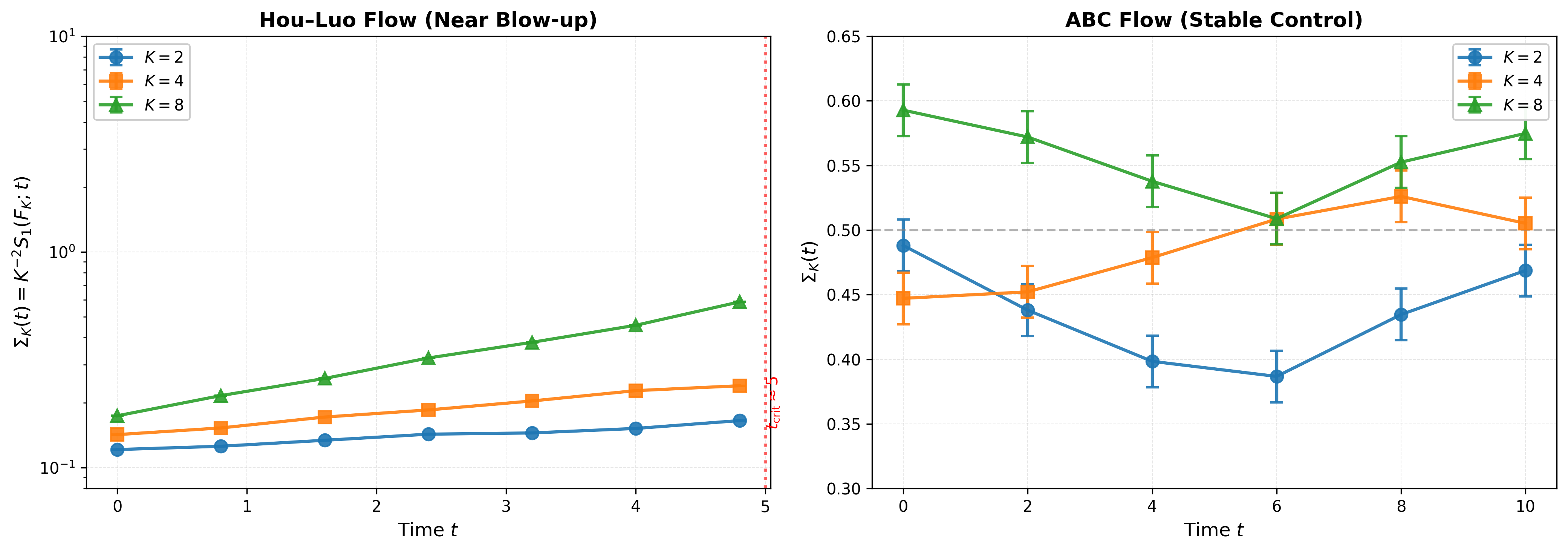}

\caption{Evolution of the scaled commutator growth $\Sigma_K(t) = K^{-2}S_1(F_K;t)$ 
for (left) Hou--Luo flow approaching conjectured blow-up at $t_{\mathrm{crit}} \approx 5$, 
and (right) ABC flow as a stable control. In the Hou--Luo case, $\Sigma_K(t)$ grows 
rapidly with $K$ as $t \to t_{\mathrm{crit}}$, indicating violation of the uniform bound 
in Theorem~3. The ABC flow maintains $\Sigma_K(t) = O(1)$ uniformly, consistent with 
global regularity. \textbf{Error bars:} standard errors of the mean, 
$\mathrm{SE} = \sigma/\sqrt{N}$ with $N=3000$ samples.}
\label{fig:sigma_K_evolution}
\end{figure}

\FloatBarrier
\paragraph{Code availability}
Python implementation of the Hou--Luo integrator, ABC flow, and $\Sigma_K$ 
computation is available online\footnote{\url{https://colab.research.google.com/drive/14m1aNAmNCDunhZOkteQtTVFkVLzPOViq?usp=sharing}}.

\section{Discussion and perspectives}\label{sec:discussion}

We have extended the operator-algebraic program for NSE in three directions:
(a) compatibility with quantum-spin representations of viscous flows \cite{MengYang2024},
(b) a computational pathway via quantum algorithms and benchmark numerics \cite{Budinski2022}, and
(c) a topological layer based on helicity and knot dynamics \cite{Moffatt1969,MoffattRicca1992,KlecknerIrvine2013}.

The conjectures in Section~\ref{sec:conjectures} are intended to be \emph{sharp enough to be falsifiable}
and \emph{structured enough to be refined} into analytically tractable statements. A next step is to
identify the precise operator $A_s(t)$ for which spectral-commutativity or determinant controls can be derived from PDE estimates, bridging the gap between the noncommutative invariants and the classical regularity criteria.


\section{Declarations}

\subsection*{Conflict of interest}
The author declares that there is no conflict of interest.
\subsection*{Data availability}
Data sharing is not applicable to this article as no datasets were generated or analysed during the current study.

\subsection*{Funding}
The author declares that no funding was received for the preparation of this manuscript.

\end{document}